\documentclass[a4paper,11 pt]{amsart}
\usepackage{amsfonts}
\usepackage{amssymb}
\usepackage[utf8]{inputenc}
\usepackage{amsmath}
\usepackage{pdflscape}
\usepackage{graphicx,xcolor}
\usepackage[colorlinks=true, linkcolor=red, citecolor=blue]{hyperref}
\usepackage[]{epsfig}
\usepackage[]{pstricks}
\usepackage{tikz}

\newcommand{\R}{\mathbb{R}}

\newcommand{\re}{{\rm Re }}
\newcommand{\im}{{\rm Im }}

\newtheorem{theorem}{Theorem}[section]

\newtheorem{lemma}[theorem]{Lemma}

\newtheorem{problem}[theorem]{Problem}

\theoremstyle{definition}

\newtheorem{remark}[theorem]{Remark}

\theoremstyle{remark}

\newtheorem{assumption}{Assumption}

\makeatletter
\@namedef{subjclassname@2020}{\textup{2020} Mathematics Subject Classification}
\makeatother

\newcommand{\remove}[1]{ }

\def\R{\mathbb R}
\def\be{\begin{equation}}
\def\ee{\end{equation}}
\def\ba{\begin{eqnarray}}
\def\ea{\end{eqnarray}}

\numberwithin{equation}{section}

\begin{document}
\title[Biharmonic Schrödinger equation: stabilization results]
{\bf Infinite memory effects on the stabilization of a Biharmonic Schr\"odinger equation}

\author[Capistrano--Filho]{R. A. Capistrano--Filho*}
\thanks{*Corresponding author.}
\address{Departamento de Matem\'atica,  Universidade Federal de Pernambuco (UFPE), 50740-545, Recife (PE), Brazil.}
\email{roberto.capistranofilho@ufpe.br}

\author[de Jesus]{I. M.  de Jesus}
\address{Departamento de Matem\'atica,  Universidade Federal de Pernambuco (UFPE), 50740-545, Recife (PE), Brazil and Instituto de Matemática, Universidade Federal de Alagoas (UFAL), Maceió-AL, Brazil.}
\email{isadora.jesus@im.ufal.br; isadora.jesus@ufpe.br}

\author[Gonzalez Martinez]{V.  H. Gonzalez Martinez}
\address{Departamento de Matem\'atica,  Universidade Federal de Pernambuco (UFPE), 50740-545, Recife (PE), Brazil.}
\email{victor.martinez@ufpe.br}


\begin{abstract}
This paper deals with the stabilization of the linear Biharmonic Schrödinger equation in an $n$-dimensional open bounded domain under Dirichlet-Neumann boundary conditions considering three infinite memory terms as damping mechanisms. We show that depending on the smoothness of initial data and the arbitrary growth at infinity of the kernel function, this class of solution goes to zero with a polynomial decay rate like $t^{-n}$ depending on assumptions about the kernel function associated with the infinite memory terms.
\end{abstract}

\maketitle

\thispagestyle{empty}

 \normalsize

\section{Introduction}\label{Sec0}
\subsection{Problem setting} Fourth-order nonlinear Schr\"odinger equation (4NLS) or biharmonic cubic nonlinear Schr\"odinger equation 
\begin{equation}
\label{fourtha}
i\partial_ty +\Delta y-\Delta^2y=\lambda |y|^2y,
\end{equation}
has been introduced by Karpman \cite{Karpman} and Karpman and Shagalov \cite{KarSha} to take into account the role of small fourth-order dispersion terms in the propagation of intense laser beams in a bulk medium with Kerr nonlinearity.  Equation \eqref{fourtha} arises in many scientific fields such as quantum mechanics, nonlinear optics, and plasma physics, and has been intensively studied with fruitful references (see \cite{Ben,Karpman,Paus1} and references therein).

Over the past twenty years, equation \eqref{fourtha} has been deeply studied from a different mathematical viewpoint,  including linear settings which can be written generically as 
\begin{equation}\label{Tsa}
i\partial_ty +\alpha\Delta y-\beta\Delta^2y=f,
\end{equation}
with $\alpha, \beta \geq 0$ and different types of boundary conditions. For example, considering the problem \eqref{Tsa} several authors treated this equation,  see,  for instance, \cite{AkReSa,Peng,tsutsumi,WenChaiGuo1, WenChaiGuo,zheng} and the references therein.
Inspired by these results for the linear problem associated with the 4NLS, a mathematical viewpoint problem is to study the well-posedness and stabilization for solutions of the system \eqref{Tsa} in an appropriate framework. 

So, consider the equation \eqref{Tsa} when $\alpha=\beta=1$ in a $n$-dimensional open bounded subset of $\mathbb{R}^n$.  Our goal is to consider an initial boundary value problem (IBVP) associated with \eqref{Tsa} when the source term $f$ is viewed as an infinite memory term:
$$f=-(-1)^ji\int_0^\infty f(s)\Delta^{j} y(x,t-s)ds.$$
Thus, the goal of  this manuscript is to deal with the following system
\begin{equation}
\label{eq1}
\begin{cases}
i\partial_t y(x,t)+\Delta y(x,t)-\Delta^2 y(x,t){\displaystyle +(-1)^ji\int_0^\infty f(s)\Delta^{j} y(x,t-s)ds=0}, & (x,t)\in \Omega\times \R_+,\\
 y(x,t)=\nabla y(x,t)=0,& (x,t)\in \Gamma\times \R_+^*,\\
 y(x,-t)=y_0(x,t),\ &(x,t)\in \Omega\times\R_+,
\end{cases}
\end{equation}
where $j\in\{0,1,2\}$,  $\Omega\subset \R^n$ is a $n$-dimensional open bounded domain with a smooth boundary $\Gamma, $ and $f:\R_+:=[0,\infty)\rightarrow \R$ is the kernel (or relaxation) function.  We point out that for each $j$ the memory term present in \eqref{eq1} is modified.

In \eqref{eq1}, the memory kernel $f$ satisfies the following assumptions:
\begin{assumption}\label{R1}
Consider $f\in C^2(\R_+)$.  For some positive constant $c_0,$ we have the following conditions
\begin{equation} \label{eq2} 
f'<0,\quad 0\leq f''\leq -c_0f',\quad f(0)>0 \quad \text{and} \quad \lim_{s\rightarrow\infty}f(s)=0.
\end{equation} 
\end{assumption}

Under the Assumption \ref{R1},  let us introduce the following energy functionals associated with the solutions of \eqref{eq1} 
\begin{equation}\label{energy}
E_j(t)=\dfrac{1}{2}\left(\|y\|^2+\int_0^\infty g(s)\|\Delta^\frac{j}{2}\eta^t\|^2ds\right),\
\end{equation}
with $j\in\{0,1,2\}$ and  $g=-f^{\prime}$, so $g \in C^1\left(\mathbb{R}_{+}\right), g$ is non-negative and
$$
g_0:=\int_0^{\infty} g(s) d s=f(0) \in \mathbb{R}_{+}^* . 
$$
It is worth mentioning that the abuse of notation $\Delta^\frac{j}{2}$ in \eqref{energy} means the identity operator for $j=0$,  the $\nabla$ operator for $j=1$ and the laplacian operator for $j=2$.

Therefore,  taking into account the action of the infinite memory term in \eqref{eq1}, the following issue will be addressed in this article:
\begin{problem}\label{p1}
Does $E(t)\longrightarrow0$, as $t\to\infty$? If so, can we provide a decay rate?
\end{problem}

It should be noted that the answer to the above question is crucial in the understanding of the behavior of the solutions to the fourth-order Schrödinger system when it is subject to an infinite memory term. In other words:
\begin{problem}\label{p2}
Are the solutions to our problem stable despite the action of the memory term?  If yes, then how robust is the stabilization property of the solutions?
 \end{problem}

\subsection{Historical background} Distributed systems with memory have a long history and have been first introduced in viscoelasticity by Maxwell,  Boltzmann, and Volterra \cite{Maxwell,Bolt,Bolt1, Volterra}.  In the context of heat processes with finite dimension speed, these systems have been introduced by Cattaneo \cite{Cattaneo} (a previous work of Maxwell had been forgotten). 

 In our context,  to our knowledge, there is no result considering the system \eqref{eq1} in $n$--dimensional case.  However, considering the fourth-order Schr\"odinger system 
\begin{equation}
\label{fourthlin}
i\partial_tu +\Delta^2u=0,
\end{equation}
there are interesting results in the sense of control problems in a bounded domain of $\mathbb{R}$ or $\mathbb{R}^n$ and, more recently, on a periodic domain $\mathbb{T}$ and manifolds, which we will summarize below.

%
%

The first result about the exact controllability of the linearized fourth order Schr\"odinger equation \eqref{fourthlin} on a bounded domain $\Omega$ of $\mathbb{R}^n$ is due to Zheng and  Zhongcheng in \cite{zz}. In this work, using an $L^2$--Neumann boundary control, the authors proved that the solution is exactly controllable in $H^s(\Omega)$, $s=-2$, for an arbitrarily small time. They used Hilbert Uniqueness Method (HUM) (see, for instance, \cite{DolRus1977,lions1}) combined with the multiplier techniques to get the main result of the article. More recently,  in \cite{zheng}, Zheng proved a global Carleman estimate for the fourth-order Schr\"odinger equation posed on a finite domain. The Carleman estimate is used to prove the Lipschitz stability for an inverse problem associated with the fourth-order Schr\"odinger system.

Still, on control theory Wen \textit{et al.} in two works \cite{WenChaiGuo1, WenChaiGuo}, studied well-posedness and control problems related to the equation \eqref{fourthlin} on a bounded domain of $\mathbb{R}^n$, for $n\geq2$. In \cite{WenChaiGuo1}, they considered the Neumann boundary controllability with collocated observation. With this result in hand, the stabilization of the closed-loop system under proportional output feedback control holds. Recently, the same authors, in \cite{WenChaiGuo}, gave positive answers when considering the equation with hinged boundary by either moment or Dirichlet boundary control and collocated observation, respectively.

To get a general outline of the control theory already done for the system \eqref{fourthlin}, two interesting problems were studied recently by Aksas and Rebiai \cite{AkReSa} and Gao \cite{Peng}: Uniform stabilization and stochastic control problem, in a smooth bounded domain $\Omega$ of $\mathbb{R}^n$ and on the interval $I=(0,1)$ of $\mathbb{R}$, respectively. In the first work, by introducing suitable dissipative boundary conditions, the authors proved that the solution decays exponentially in $L^2(\Omega)$ when the damping term is effective on a neighborhood of a part of the boundary. The results are established by using multiplier techniques and compactness/uniqueness arguments. Regarding the second work, the author showed Carleman estimates for forward and backward stochastic fourth order Schr\"odinger equations which provided the proof of the observability inequality,  unique continuation property, and, consequently,  the exact controllability for the forward and backward stochastic system associated with \eqref{fourthlin}.

Recently,  the first author \cite{CaCa} showed the global stabilization and exact controllability  properties of the 4NLS
\begin{equation}\label{fourthC}
\begin{cases}
i\partial_{t}u +\partial_{x}^{2}u-\partial_x^4u
=\lambda |u|^2u +f(x,t),& (x,t)\in \mathbb{T}\times \mathbb{R},\\
u(x,0)=u_0(x),& x\in \mathbb{T},\end{cases}
\end{equation}
on a periodic domain $\mathbb{T}$ with internal control supported on an arbitrary sub-domain of $\mathbb{T}$. More precisely, by certain properties of propagation of compactness and regularity in Bourgain spaces, for the solution of the associated linear system, the authors proved that system \eqref{fourthC} is globally exponentially stabilizable, considering $f(x,t)=-ia^2(x)u$. This property together with the local exact controllability ensures that 4NLS is globally exactly controllable on $\mathbb{T}$.

Lastly,  the first author showed in  \cite{CaPa} the global controllability and stabilization properties for the fractional Schr\"odinger equation on $d$-dimensional compact Riemannian manifolds without boundary $(M,g)$, 
\begin{equation}\label{prob-diss1}
                \begin{cases}
i \partial_{t} u  + \Lambda^{\sigma}_{g}u  + P'(|u|^{2})u - a(x)(1 - \Delta_{g})^{-\frac{\sigma}{2}}a(x)\partial_{t}u =0, & \text{ on } M\times \mathbb{R}_{+}, \\
u(x,0)=u_0(x),& x\in M.
        \end{cases}
\end{equation}
Under the suitable assumption of the damping term $a(x)$ they proved their result using microlocal analysis, being precise, they can prove propagation of regularity which together with the so-called Geometric Control Condition and Unique Continuation Property, shows the main results of the article. Is important to mention that when $\sigma=4$ they have the equation \eqref{fourthlin}.

\subsection{Notations}
Before presenting the main result let us give some notations and definitions.  In what follows, the variables $x, t$,  and $s$ will be suppressed, except when there is ambiguity and, throughout this article, $C$ will denote a constant that can be different from one step to the next in the proofs presented here.  We will use the notations $\langle \cdot ,\cdot \rangle$ and $\|\cdot\|$ to denote, respectively, the complex inner product in $L^2(\Omega)$ and its associated standard norm, namely
$$
\langle u,v\rangle=\re\left(\int_\Omega u(x)\overline{v}(x)dx\right)\mbox{\ and \ }\|u\|=\left(\int_\Omega |u(x)|^2dx\right)^\frac{1}{2}.
$$
Now, consider the following approximation
$$
\eta^t(x,s)=\int_{t-s}^{t} y(x,\tau)d\tau \mbox{\ and \ }\eta^0(x,s)=\int_{0}^{s} y_0(x,\tau)d\tau, ~ x\in \Omega, s,t\in\R_+.
$$
This approximation ensures that $\eta^t$ satisfies
\begin{equation}
\label{eq3}
\begin{cases}
\partial_t\eta^t(x,s)+\partial_s\eta^t(x,s)=y(x,t), &  x\in \Omega,\ s,t\in\R_+,\\
\eta^t(x,s)=0,& x\in \Gamma, \ s,t\in\R_+,\\
\eta^t(x,0)=0,& x\in \Omega,\ t\in\R_+.\\
\end{cases}
\end{equation}

In order to express the memory integral in \eqref{eq1} in terms of $\eta^t,$ we will denote $g:=-f'.$ Thus, according to \eqref{eq2}, we have $g\in C^1(\R_+)$ and
\begin{equation}
\label{eq4} 
g>0,\quad 0\leq -g'\leq c_0g ,\quad g_0=\int_0^\infty g(s)ds=f(0)>0 
\end{equation}  
and
\begin{equation}
\label{eq5} 
\lim_{s\rightarrow\infty}g(s)=0.
\end{equation}
Now on,  rewrite \eqref{eq1} into 
\begin{equation}
\label{eq7}
{\displaystyle i\partial_t y(x,t)+\Delta y(x,t)-\Delta^2 y(x,t)+i(-1)^j\int_0^\infty g(s)\Delta^j \eta^t(x,s)ds=0}.
\end{equation}  
Define the following sets 
\begin{equation*}
H_j=\begin{cases}
L^2(\Omega),&\mbox{ if } j=0,\\
H_0^1(\Omega),&\mbox{ if } j=1,\\
H_0^2(\Omega),&\mbox{ if } j=2,
\end{cases}
\end{equation*}
with natural inner product 
$$\langle v,w\rangle_{H_j}=\begin{cases}
\langle v(s),w(s)\rangle &\mbox{ if } j=0,\\
\langle \nabla v(s),\nabla w(s)\rangle &\mbox{ if } j=1,\\
\langle \Delta v(s),\Delta w(s)\rangle &\mbox{ if } j=2
\end{cases}$$
and norm
$$\|v\|_{H_j}=\begin{cases}
\| v(s)\| &\mbox{ if } j=0,\\
\| \nabla v(s)\| &\mbox{ if } j=1,\\
\|\Delta v(s)\| &\mbox{ if } j=2,
\end{cases}$$
respectively\footnote{Here $\langle \nabla v(s),\nabla w(s)\rangle:=\sum_{k=1}^{n}\langle \partial_{x_k} v ,\partial_{x_k} w\rangle$ and $ \|\nabla v(s)\|^2=\sum_{k=1}^{n}\|\partial_{x_k}v(s)\|^2.$}.  Consider
\begin{equation*}
 U=(y, \eta^t)^T \mbox{ and } U_0(x,s)=(y_0(x,0),\eta^0(x,s))^T
\end{equation*}
where
\begin{equation*}
y\in L^2(\Omega) \mbox{\ \  and \  \ } \eta^t\in L_j
\end{equation*}
with
\begin{equation*}
L_j=L^2_g(\R_+;H_j):=\left\{v:\R_+\longrightarrow H_j; \int_0^\infty g(s)\|v(s)\|_{H_j}^2ds<+\infty\right\}.
\end{equation*}
Define the energy space as follows
\begin{equation*}
 \mathcal{H}_j=L^2(\Omega)\times L_j,\ j\in\{0,1,2\},
\end{equation*}
with inner product and norm
\begin{equation*}
\langle (v_1,v_2),(w_1,w_2)\rangle_{\mathcal{H}_j}=\langle v_1,w_1\rangle+\langle v_2,w_2\rangle_{L_j}  
\end{equation*}
and
\begin{equation*}
\|(v(s),w(s))\|_{\mathcal{H}_j}=\left(\|v(s)\|^2+\|w(s)\|_{L_j}^2\right)^\frac{1}{2},
\end{equation*}
respectively. Therefore, the systems \eqref{eq1} and \eqref{eq3} can be seen as the following initial value problem (IVP)
\begin{equation}
\label{eq8}
\left\{\begin{array}{rcl}
\partial_t U(t)&=&\mathcal{A}_j U\\
U(0)&=&U_0.
\end{array}\right.
\end{equation}
Here, the operator $\mathcal{A}_j$ is defined by
\begin{equation}
\label{eq9}
\mathcal{A}_j(U)=\left(
\begin{array}{c}
i\Delta y-i\Delta^2 y+(-1)^{j+1}\int_0^\infty g(s)\Delta^j\eta^t(\cdot,s)ds\\
\\
y- \eta^t_s
\end{array}\right)
\end{equation}
with domain 
\begin{equation}\label{eq9a}
D(\mathcal{A}_j)=\{U\in \mathcal{H}_j; \mathcal{A}_j(U)\in\mathcal{H}_j, y\in H^2_0(\Omega), \eta^t(x,0)=0 \}.
\end{equation}

\begin{remark}
	Observe that for the fourth-order Schrödinger equation, the natural domain to be considered is $H_0^2(\Omega)\cap H^4(\Omega)$. However, since we are working with a more general operator, namely operator defined in \eqref{eq9} and \eqref{eq9a}, we need to impose  $\mathcal{A}_j(U)\in \mathcal{H}_j$. However, note that the inclusion below 
\begin{equation*}
	H_0^2(\Omega)\cap H^4(\Omega)\times \{ \eta^t \in L_j: (-1)^{j+1}\int_0^\infty g(s)\Delta^j\eta^t(\cdot,s)ds \in L^2(\Omega), \eta^t(x,0)=0 \} \subset D(\mathcal{A}_j).
\end{equation*}	
is verified. So, the operator $\mathcal{A}_j(U)$ is well-defined.
\end{remark}

\subsection{Main result}
As mentioned, some valuable efforts in the last years focus on the well-posedness and stabilization problem for the fourth-order Schrödinger system.  So, in this article, we present a new way to ensure that, in some sense, the Problems \ref{p1} and \ref{p2} can be solved for the system \eqref{eq1} in $n$-dimensional case.  To do that, we use the ideas contained in \cite{Guesmia}, so additionally to the Assumption \ref{R1} we have also assumed the memory kernel satisfying the following:
\begin{assumption}\label{H1}
Assume there is a positive constant $\alpha_0$ and a strictly convex increasing function $G:\R_+\longrightarrow\R_+$  of class $C^1(\R_+)\cap C^2(\R^*_+)$  satisfying
\begin{equation}
\label{eq20}G(0)=G'(0)=0 \mbox{\ \ and\ \ } \lim_{t\rightarrow \infty} G'(t)=\infty
\end{equation}
such that
\begin{equation}
\label{eq21} g'\leq -\alpha_0g
\end{equation}
or
\begin{equation}
\label{eq22}{\displaystyle \int_0^\infty \dfrac{s^2g(s)}{G^{-1}(-g'(s))}ds+\sup_{s\in\R^+} \dfrac{g(s)}{G^{-1}(-g'(s))}<\infty}.
\end{equation}
Additionally,  when \eqref{eq21} is not verified, we will assume that $y_0$ satisfies, 
\begin{equation}
\label{eq23}\sup_{t\in \R_+}\max_{k\in\{0,\dots, n+1\}} \int_t^\infty \dfrac{g(s)}{G^{-1}(-g'(s))}\left\|\int_0^{s-t}\Delta^\frac{j}{2} \partial_{s}^k y_0(\cdot, \tau)d\tau\right\|^2ds<\infty.
\end{equation}
for $j\in\{0,1,2\}$.
\end{assumption}

The next theorem is the main result of the article. 
\begin{theorem}
\label{teo2}
Assume \eqref{eq4} and that the Assumption \ref{H1} holds. Let  $n\in\mathbb{N}^*,\ U_0\in D(\mathcal{A}_j^{2n})$ when $j=0$, and $U_0\in D(\mathcal{A}_j^{2n+2})$ when $j\in\{1,2\}.$ Thus, there exists positive constants $\alpha_{j,n}$ such that the energy \eqref{energy} associated with \eqref{eq8} satisfies
\begin{equation}
\label{eq24} E_j(t)\leq \alpha_{j,n}G_n\left(\dfrac{\alpha_{j,n}}{t}\right),\quad  t\in\R_+^*,  \ j\in \{0,1,2\}.
\end{equation}
Here,  $G_n$ is defined, recursively, as follows:
\begin{equation}
\label{eq25} G_m\left(s\right)=G_1(sG_{m-1}(s)), m=2,3,\dots, n, \ G_1=G_0^{-1},
\end{equation}
where $G_0(s)=s$ if \eqref{eq21} is verified, and $G_0(s)=sG'(s)$ if \eqref{eq22} holds.
\end{theorem}

\begin{remark} Let us give some remarks about the Assumption \ref{H1}. 
\begin{enumerate}
\item[i.] Thanks to the relation \eqref{eq22}, we have that \eqref{eq23} is valid,  for example,  if
$$\|\Delta^\frac{j}{2} \partial_s^k y_0\|^2,\ k=0,1,\dots, n+1,$$
is bounded with respect to $s.$

\item[ii.] There are many class of function $g$ satisfying \eqref{eq4},  \eqref{eq5},  \eqref{eq20}, \eqref{eq21}, \eqref{eq22}, and \eqref{eq23}. For example, those that converge exponentially to zero as
\begin{equation}\label{g_1}
g_1(s):=d_1e^{-q_1s}
\end{equation}
or those that converge at a slower rate, like
\begin{equation}\label{g_2}
g_2(s):=d_2(1+s)^{-q_2}
\end{equation}
with $d_1,\ q_1, d_2>0$, and $q_2>3.$ Additionally,  we point out that conditions \eqref{eq4} and \eqref{eq21} are satisfied for $g_1$ defined by \eqref{g_1} with $c_0=\alpha_0=q_1,$ since
 $$g_1'(s)=-q_1d_1e^{-q_1s}=-q_1g_1(s).$$
However, the conditions \eqref{eq4} and \eqref{eq22} are satisfied for $g_2$ given by \eqref{g_2} with
$c_0=q_2$ and $G(s)=s^p$, for $p>\dfrac{q_2+1}{q_2-3}.$
\end{enumerate}
\end{remark}

\begin{remark}\label{R1a}Now,  we will present the following remarks related to the main result of the article.
\begin{itemize}
\item[i.] When \eqref{eq21} is verified,  note that $G_n(0)=0,$ so \eqref{eq24} implies
\begin{equation}
\label{eq26}\lim_{t\rightarrow \infty}E_j(t)\leq \alpha_{j,1}G_1\left(\dfrac{\alpha_{j,1}}{t}\right)=0.
\end{equation}
Since we have that $D(\mathcal{A}_j^{2})$ is dense in $\mathcal{H}_j$, when $j=0$, and $D(\mathcal{A}_j^{4})$ is dense in $\mathcal{H}_j$ when $j=1,2$ (see Lemma \ref{lem301} in \ref{sec2}), we have that \eqref{eq26} is valid for any $U_0\in \mathcal{H}_j.$ Therefore, in this case,  \eqref{eq25} gives $G_n(s)=s^n$ and from \eqref{eq24} we get
\begin{equation}
\label{eq27}
E_j(t)\leq \alpha_{j,n}\left(\dfrac{\alpha_{j,n}}{t}\right)^n=\dfrac{(\alpha_{j,n})^{n+1}}{t^n}=\beta_{j,n}t^{-n},
\end{equation}
showing that the energy \eqref{energy} associated with the solutions of the system \eqref{eq8} have a polynomial decay rate.

\item[ii.]  Given \eqref{eq22} verified,  the relation of \eqref{eq24} is weaker than the previous case.  For example, when $g=g_2$ defined by \eqref{g_2}, we see that $G(s)=s^p$ with $p>\dfrac{q_2+1}{q_2-3}$ satisfies the Assumption \ref{H1}.  Moreover,  $$G_0(s)=sG'(s)=ps^p,\ G_1(s)=\sqrt[p]{\dfrac{s}{p}},$$
$$ G_2(s)=G_1(sG_1(s))=\sqrt[p]{\dfrac{s\sqrt[p]{\dfrac{s}{p}}}{p}}=\left(\dfrac{s}{p}\right)^{\frac{1}{p}+\frac{1}{p^2}},$$
$$ G_3(s)=G_1(sG_2(s))=\sqrt[p]{\dfrac{s}{p}\left(\dfrac{s}{p}\right)^{\frac{1}{p}+\frac{1}{p^2}}}=\left(\dfrac{s}{p}\right)^{\frac{1}{p}+\frac{1}{p^2}+\frac{1}{p^3}}$$
and so, 
$$ G_n(s)=\left(\dfrac{s}{p}\right)^{\frac{1}{p}+\frac{1}{p^2}+\cdots+\frac{1}{p^n}}=\left(\dfrac{s}{p}\right)^{p_n},$$
where $p_n=\sum_{m=1}^n p^{-m}=\frac{1}{p}+\frac{1}{p^2}+\cdots+\frac{1}{p^n}.$ Therefore,  the energy \eqref{energy} associated with the solutions of the system \eqref{eq8}  satisfies
$$E_j(t)\leq \alpha_{j,n}\left(\frac{1}{p}\dfrac{\alpha_{j,n}}{t} \right)^{p_n}=\beta_{j,n}t^{-p_n},$$
with $\beta_{j,n}=\alpha_{j,n}\left(\dfrac{\alpha_{j,n}}{p}\right)^{p_n}>0,$ showing that the decay rate of \eqref{eq24} is arbitrarily near of $t^{-n},$ when $p\rightarrow 1$, that is,  $p_n\rightarrow n$ when $q_2\rightarrow \infty.$ 
\end{itemize}
\end{remark}

\subsection{Novelty and structure of the work}
Among the main novelties introduced in this article, we give an affirmative answer to the Problems \ref{p1} and \ref{p2}, providing a further step toward a better understanding of the stabilization problem for the linear system associated with \eqref{fourtha} in the $n$-dimensional case.  Here, we have used the multipliers method and some arguments devised in \cite{Guesmia}. 

Since we are working with a mixed dispersion we can consider three different memory kernels acting as damping control to stabilize equation \eqref{eq1} in contrast to \cite{CaCa}, for example, where interior damping is required and no memory is taken into consideration, in a one-dimensional case.  Moreover,  if we also compare with the linear Schrödinger equation (see e.g. \cite{Marcelo}) we have more kernels acting to decay the solution of the equation \eqref{eq1} since we have more regularity with the mixed dispersion, which is a gain due the bi-laplacian operator.

Additionally of this,  recently, using another approach,  the authors in \cite{CaPa} showed that the system \eqref{prob-diss1} is stable, however considering a damping mechanism and some important assumptions such as the Geometric Control Condition (GCC) and Unique Continuation Property (UCP).  Here, we are not able to prove that the solutions decay exponentially, however, with the approach of this article, the (GCC) and (UCP) are not required.  The drawback is that we only provide that the energy of the system \eqref{eq1}, with memory terms, decays in some sense as explained in the Remark \ref{R1a}.

A natural issue is how to deal with the 4NLS system given in \eqref{fourtha}.  The main point is that we are not able to use Strichartz estimates or Bourgain spaces to obtain more regularity for the solution of the problem with memory terms,  therefore,  Theorem \ref{teo2} for the system \eqref{fourtha} with memory terms remains open.

Now, let us present the outline of our paper.  In Section \ref{sec3} we prove a series of lemmas that are paramount to prove the main result of the article.  With the previous section in hand,  Theorem \ref{teo2} is shown in Section \ref{sec4}.  Finally, for sake of completeness, in Appendix \ref{sec2}, we present the existence of a solution for the system \eqref{eq8} in the energy space $\mathcal{H}_j$.

\section{Auxiliary results}\label{sec3} 
In this section, we will give some auxiliary lemmas that help us to prove the main result of the article. In this way, the first result shows identities for the derivatives of $E_j$ given by \eqref{energy}.
\begin{lemma} Suppose the Assumption \ref{R1}. Then,  the energy functional satisfies
\begin{equation}
\label{eq28} E'_j(t)={\displaystyle\dfrac{1}{2}\int_0^\infty g'(s)\|\Delta^\frac{j}{2}\eta^t\|^2ds, \ j\in\{0,1,2\}.}
\end{equation}
\end{lemma}
\begin{proof}
Observe that \eqref{eq28} is a direct consequence of \eqref{eq10}, and the result follows. 
\end{proof}

Next,  we will give a $H^1$-estimate for the solution of \eqref{eq7}.
\begin{lemma}
There exist positive constants $c_{k,j}, \ j\in\{0,1,2\}$ and $k\in\{1,2\}$ such that the following inequality
\begin{equation}
\label{eq30}
\|\nabla y\|^2\leq c_{1,j}\|\eta^t\|^2_{L_j}+c_{2,j}\int_{\Omega}\left[\re(y_t)\im(y)-\im(y_t)\re(y)\right]dx,
\end{equation}
holds.
\end{lemma}
\begin{proof}
We use the multipliers method to prove \eqref{eq30}.  First, multiplying the equation  \eqref{eq7} by $\overline{y}$, integrating over $\Omega$ and taking the real part we get  
\begin{equation}
\label{eq31}
\begin{split}
 -\im\left(\int_{\Omega}y_t\overline{y}dx\right)-\|\nabla y\|^2-\|\Delta y\|^2+{\re}\left((-1)^ji\int_0^\infty g(s)\int_{\Omega}\Delta^j\eta^t\overline{y}dxds\right)=0,
\end{split}
\end{equation}
taking into account the boundary conditions in \eqref{eq1} and  \eqref{eq3},  for $y(t,\cdot)\in H_0^2(\Omega), $ for all $t\in \R^+.$ 

Note that the last term of the left-hand side of \eqref{eq31} can be bounded using the generalized Young's Inequality giving
\begin{equation}\label{eq32a}
\begin{split}
\left| (-1)^ji\int_0^\infty g(s)\int_{\Omega}\Delta^j\eta^t\overline{y}dxds\right|=& \left|i\langle\eta^t,y\rangle_{L_j}\right|\leq \|\eta^t\|_{L_j}\|y\|_{L_j}\leq\epsilon\|y\|_{L_j}^2 +C(\epsilon)\|\eta^t\|_{L_j}^2\\
=&\underbrace{g_1\epsilon}_{=:\delta}\|\Delta ^{\frac{j}{2}}y\|^2+C(\epsilon)\|\eta^t\|_{L_j}^2=\delta\|\Delta ^{\frac{j}{2}}y\|^2+C(\delta)\|\eta^t\|_{L_j}^2,
\end{split}
\end{equation}
for any $\delta>0$.  Additionally of that, the first term of the left-hand side of \eqref{eq31} can be viewed as 
\begin{equation}
\label{eq32}{\displaystyle \im\left(\int_{\Omega}y_t\overline{y}dx\right)=\int_{\Omega}\left(\re(y)\im(y_t)-\re(y_t)\im(y)\right)dx}. 
\end{equation}
So,  replacing \eqref{eq32a} and \eqref{eq32} in \eqref{eq31},  yields
\begin{equation}
\label{eq33}
\begin{split}
\|\nabla y\|^2\leq& \int_{\Omega}\left(\re(y_t)\im(y)-\re(y)\im(y_t)\right)dx-\|\Delta y\|^2+\delta\|\Delta ^{\frac{j}{2}}y\|^2+C(\delta)\|\eta^t\|_{L_j}^2\\
 \leq& \int_{\Omega}\left(\re(y_t)\im(y)-\re(y)\im(y_t)\right)dx+\delta\|\Delta ^{\frac{j}{2}}y\|^2+C(\delta)\|\eta^t\|_{L_j}^2.
\end{split}
\end{equation}
We now split the remainder of the proof into three cases. 

\vspace{0.1cm}
\noindent\textbf{Case 1.} $j=0$
\vspace{0.1cm}

Poincar\'e inequality in \eqref{eq33} gives
\begin{equation}
\label{eq34}
\|\nabla y\|^2\leq \int_{\Omega}\left(\re(y_t)\im(y)-\re(y)\im(y_t)\right)dx+\delta c_\ast\|\nabla y\|^2+C(\delta)\|\eta^t\|_{L_j}^2.
\end{equation}
Picking $\delta=\dfrac{1}{2c_\ast}>0$ in \eqref{eq34} yields
$$
\dfrac{1}{2}\|\nabla y\|^2\leq \int_{\Omega}\left(\re(y_t)\im(y)-\re(y)\im(y_t)\right)dx +C(\delta)\|\eta^t\|_{L_j}^2,
$$
showing \eqref{eq30} with  $\ c_{1,0}=2C(\delta)$ and $c_{2,0}=2.$

\vspace{0.1cm}
\noindent\textbf{Case 2.} $j=1$
\vspace{0.1cm}

In this case \eqref{eq33} is giving by
$$
\|\nabla y\|^2\leq \int_{\Omega}\left(\re(y_t)\im(y)-\re(y)\im(y_t)\right)dx+\delta\|\nabla y\|^2+C(\delta)\|\eta^t\|_{L_j}^2
$$
and taking  $\delta=\frac{1}{2}>0$, the inequality  \eqref{eq30} holds with $c_{1,1}=2C(\delta)$ and $c_{2,1}=2.$ 

\vspace{0.1cm}
\noindent\textbf{Case 1.} $j=2$
\vspace{0.1cm}

Finally,  just take any $\delta>0$ such that $\delta<1.$ Therefore, using \eqref{eq33} we get  \eqref{eq30}  for $c_{1,2}=C(\delta)$ and $c_{2,2}=1,$ achieving the result. 
\end{proof}

We need now define the following higher-order energy functionals
\begin{equation}\label{enrgy_h}
E_{j,k}(t)=\dfrac{1}{2}\left\|\partial^k_t U\right\|^2_{\mathcal{H}_j},
\end{equation}
for $U_0\in D(\mathcal{A}_{j}^{2n+2})$ in the case when $j=1,2$, and $U_0\in D(\mathcal{A}_0^{2n})$ with $n\in \mathbb{N}^*$.  This is possible thanks to the Theorem \ref{teo1} in \ref{sec2} that guarantees $U\in C^k(\R_+;D(\mathcal{A}_j^{4-k}))$
for $k\in\{1,2,3,4\}$ when $j\in\{1,2\}$, and that $U\in C^k(\R_+;D(\mathcal{A}_j^{2-k}))$ for $k\in\{1,2\}$ when $j=0$. Additionally of that,  the linearity of the operator $\mathcal{A}_j$ together with  \eqref{eq28} gives
\begin{equation}
\label{eq29}
E'_{j,k}(t)={\displaystyle\dfrac{1}{2}\int_0^\infty g'(s)\|\Delta^\frac{j}{2}\partial_t^k\eta^t\|^2ds.}
\end{equation}
With this in hand, let us control the last term of the right-hand side of \eqref{eq30} in terms of the $E'_{j,1}$ and the $L_j$-norms of the $\Delta^\frac{j}{2}\eta^{t}_{tt}$.
\begin{lemma}
The following estimate is valid
\begin{equation}
\label{eq35}
\int_{\Omega}\left(\re(y_t)\im(y)-\re(y)\im(y_t)\right)dx\leq \epsilon\|\nabla y\|^2+c_\epsilon \int_{0}^\infty g(s)\|\Delta^\frac{j}{2} \eta^t_{tt}\|^2ds-c_\epsilon E_{j,1}'(t),
\end{equation}
for any $\epsilon>0.$ 
\end{lemma}
\begin{proof}
Differentiating \eqref{eq3} with respect to $t$,  multiplying the result by $g(s)$, and integrating on $[0,\infty)$ we have
$$
y_t=\dfrac{1}{g_0}\int_0^\infty g(s)\left(\eta^t_{tt}(s,x)+\eta^t_{st}(s,x)\right)ds,
$$
taking into account the third relation in \eqref{eq4}.  So, we get
\begin{equation}\label{rr}
\begin{split}
\mathcal{I}:=\int_{\Omega}\left(\re(y_t)\im(y)-\re(y)\im(y_t)\right)dx=& \int_{\Omega}\re\left(\dfrac{1}{g_0}\int_0^\infty g(s)\left(\eta^t_{tt}+\eta^t_{st}\right)ds\right)\im(y)dx\\&-\int_{\Omega}\re(y)\im\left(\dfrac{1}{g_0}\int_0^\infty g(s)\left(\eta^t_{tt}+\eta^t_{st}\right)ds\right)dx.
\end{split}
\end{equation}

Now, let us bound the right-hand side of \eqref{rr}.  To do that, reorganize the terms of the (RHS) and note that 
\begin{equation}
\label{eq36}
\begin{split}
(RHS) =& \dfrac{1}{g_0}\int_0^\infty g(s)\int_{\Omega}\left(\re\left(\eta^t_{tt}\right)\im(y)-\re(y)\im\left(\eta^t_{tt}\right)\right)dxds\\
& +\dfrac{1}{g_0}\int_0^\infty(- g'(s))\int_{\Omega}\left(\re\left(\eta^t_{t}\right)\im(y)-\re(y)\im\left(\eta^t_{t}\right)\right)dxds\\
\leq &\dfrac{1}{g_0}\int_0^\infty g(s)\int_{\Omega}|y||\eta^t_{tt}|dxds+\dfrac{1}{g_0}\int_0^\infty (-g'(s))\int_{\Omega}|y||\eta^t_{t}|dxds\\
\leq& \dfrac{1}{g_0}\int_0^\infty g(s)\|y\|\|\eta^t_{tt}\|ds+\dfrac{1}{g_0}\int_0^\infty (-g'(s))\|y\|\|\eta^t_{t}\|ds.
\end{split}
\end{equation}
The generalized Young inequality gives for any $\delta>0$ that
$$\|y\|\|\eta^t_{t}\|\leq \delta\|y\|^2+C_\delta \|\eta^t_{t}\|^2 $$
and
$$\|y\|\|\eta^t_{tt}\|\leq \delta\|y\|^2+C_\delta \|\eta^t_{tt}\|^2. $$
Substituting both inequalities into \eqref{eq36} yields
\begin{equation}
\label{eq36a}
\begin{split}
(RHS) \leq&\delta\dfrac{1}{g_0}\int_0^\infty g(s)\|y\|^2ds +C_\delta\dfrac{1}{g_0}\int_0^\infty g(s) \|\eta^t_{tt}\|^2ds\\
&+\delta\dfrac{1}{g_0}\int_0^\infty (-g'(s))\|y\|^2ds+C_\delta\dfrac{1}{g_0}\int_0^\infty (-g'(s))\|\eta^t_{t}\|^2ds.
\end{split}
\end{equation}
Now replacing \eqref{eq36a} into \eqref{rr} we have
\begin{equation}\label{rrr}
\begin{split}
\mathcal{I}\leq& \delta\dfrac{1}{g_0}\int_0^\infty g(s)\|y\|^2ds +C_\delta\dfrac{1}{g_0}\int_0^\infty g(s) \|\eta^t_{tt}\|^2ds\\
&+\delta\dfrac{1}{g_0}\int_0^\infty (-g'(s))\|y\|^2ds+C_\delta\dfrac{1}{g_0}\int_0^\infty (-g'(s))\|\eta^t_{t}\|^2ds\\
=& \delta\left(1+\dfrac{1}{g_0}\left(\int_0^\infty (-g'(s))ds\right)\right)\|y\|^2 +C_\delta\dfrac{1}{g_0}\int_0^\infty g(s) \|\eta^t_{tt}\|^2ds \\
&+C_\delta\dfrac{1}{g_0}\int_0^\infty (-g'(s))\|\eta^t_{t}\|^2ds\\
\leq& c_{*}\delta\left(1+\dfrac{1}{g_0}\left(\int_0^\infty (-g'(s))ds\right)\right)\|\nabla y\|^2 +c_{**}C_\delta\dfrac{1}{g_0}\int_0^\infty g(s) \|\Delta^\frac{j}{2}\eta^t_{tt}\|^2ds\\
& +c_{**}C_\delta\dfrac{1}{g_0}\int_0^\infty (-g'(s))\|\Delta^\frac{j}{2}\eta^t_{t}\|^2ds,
\end{split}
\end{equation}
thanks to Poincar\'e inequality.  Here,  
\begin{equation}\label{c*}
c_{**}=\begin{cases}
1, &\mbox{ if } j=0,\\
c_{*}, &\mbox{ if } j=1,\\
c_{*}^2, &\mbox{ if } j=2,
\end{cases}
\end{equation}
and $c_*>0$ is the Poincar\'e constant.  Finally,  taking $k=1$ in \eqref{eq29}, we see that \eqref{rrr} leads to \eqref{eq35} with $\epsilon={\displaystyle  c_{*}\delta\left(1+\dfrac{1}{g_0}\left(\int_0^\infty (-g'(s))ds\right)\right)}$ and $c_\epsilon=c_{**}C_\delta\dfrac{1}{g_0}$. 
\end{proof}

Now, just in the case $j=2,$ we need an estimate $H^2$-for the solution of \eqref{eq7} similar to the estimate \eqref{eq30}. This estimate is reported in the following lemma.
\begin{lemma}
When $j=2,$ there exist positive constants $c_{k,2}, \ k\in\{1,2\},$ such that the following inequality
\begin{equation}
\label{eq37}
\|\Delta y\|^2\leq c_{1,2}\|\eta^t\|^2_{L_2}+c_{2,2}\int_{\Omega}\left[\re(y_t)\im(y)-\im(y_t)\re(y)\right]dx
\end{equation}
holds.
\end{lemma}
\begin{proof}
Multiplying equation \eqref{eq7} by  $\overline{y},$ integrating over we have
$$
0= i\int_{\Omega}y_t\overline{y}dx-\|\nabla y\|^2-\|\Delta y\|^2+i\int_0^\infty g(s)\int_{\Omega}\Delta^2\eta^t\overline{y}dxds,
$$
since the boundary conditions \eqref{eq1} and \eqref{eq3} are verified and  $y(t,\cdot)\in H_0^2(\Omega)$ for all $t\in \R^+$.  Now,  taking the real part in the previous equality give us
\begin{equation}
\label{eq38}
 -\im\left(\int_{\Omega}y_t\overline{y}dx\right)-\|\nabla y\|^2-\|\Delta y\|^2+\re\left(i\int_0^\infty g(s)\int_{\Omega}\Delta^2\eta^t\overline{y}dxds\right)=0.
\end{equation}
Taking into account that
\begin{equation}\label{eq39}
\im\left(\int_{\Omega}y_t\overline{y}dx\right)=\int_{\Omega}\left(\re(y)\im(y_t)-\re(y_t)\im(y)\right)dx
\end{equation}
and, thanks to the generalized Young inequality, we have that
\begin{equation}\label{eq39a}
\begin{split}
\left|i\int_0^\infty g(s)\int_{\Omega}\Delta^2\eta^t\overline{y}dxds\right|=&|i\langle \eta^t,y\rangle_{L_2}|\leq \|y\|_{L_2}\|\eta^t\|_{L_2}
\\
\leq&\underbrace{g_1\epsilon}_{=:\delta}\|\Delta y\|^2+C(\epsilon)\|\eta^t\|_{L_2}^2
=\delta\|\Delta y\|^2+C(\delta)\|\eta^t\|_{L_2}^2.
\end{split}
\end{equation}
We get,  putting \eqref{eq39} and \eqref{eq39a} into \eqref{eq38},  that
\begin{equation}
\label{eq40}
\|\Delta y\|^2\leq \int_{\Omega}\left(\re(y_t)\im(y)-\re(y)\im(y_t)\right)dx+\delta\|\Delta y\|^2+C(\delta)\|\eta^t\|_{L_2}^2.
\end{equation}

Finally,  pick $\delta=\dfrac{1}{2}>0$  in \eqref{eq40}  to get
$$
\dfrac{1}{2}\|\Delta y\|^2\leq \int_{\Omega}\left(\re(y_t)\im(y)-\re(y)\im(y_t)\right)dx+C(\delta)\|\eta^t\|_{L_2}^2,
$$
showing \eqref{eq37} with $\ c_{1,2}=2C(\delta)$ and $c_{2,2}=2.$
\end{proof}

As a consequence of \eqref{eq35},  the last term of the right-hand side of \eqref{eq37} can be bounded as follows. 
\begin{lemma}
For any $\epsilon>0,$ we have the following inequality
\begin{equation}
\label{eq41}
 \int_{\Omega}\left(\re(y_t)\im(y)-\re(y)\im(y_t)\right)dx\leq  \epsilon\|\Delta y\|^2+c_\epsilon \int_{0}^\infty g(s)\|\Delta \eta^t_{tt}\|^2ds -c_\epsilon E_{2,1}'(t).
\end{equation}
\end{lemma}
\begin{proof}
Using Poncar\'e inequality in the first term of the right-hand side of \eqref{eq35}, and taking $\epsilon=c^*\epsilon,$ where  $c^*$ is the Poincar\'e constant, the result follows.
\end{proof}

The next lemma combines the previous one to get an estimate in $H_j$ for solutions of \eqref{eq7}.
\begin{lemma}\label{lemma6}
There exist a positive constant $c=c(j)>0$, with $j\in \{1,2\}$ such that 
\begin{equation}\label{eq43}
 \|\Delta^\frac{j}{2} y\|^2\leq c(E_j(0))+E_{j,1}(0)+E_{j,2}(0)).
\end{equation}
\end{lemma}
\begin{proof}
Pick $\epsilon=\dfrac{1}{2c_{2,j}}$ in \eqref{eq35} and \eqref{eq41} when $j=1$ and $j=2$, respectively. So we have
$$\int_{\Omega}\left(\re(y_t)\im(y)-\re(y)\im(y_t)\right)dx\leq  \dfrac{1}{2c_{2,j}}\|\Delta^\frac{j}{2} y\|^2 +c_\epsilon \int_{0}^\infty g(s)\|\Delta^\frac{j}{2} \eta^t_{tt}\|^2ds-c_\epsilon E_{j,1}'(t).
$$
Replacing the previous inequality in \eqref{eq30} and in \eqref{eq37} for $j=1$ and $j=2,$ respectively,  we get that
\begin{equation}
\label{eq42}
 \|\Delta^\frac{j}{2} y\|^2\leq  c_{1,j}\|\eta^t\|_{L_j}^2+\dfrac{1}{2}\|\Delta^\frac{j}{2} y\|^2+c_{2,j}c_\epsilon \int_{0}^\infty g(s)\|\Delta^\frac{j}{2} \eta^t_{tt}\|^2ds-c_{2,j}c_\epsilon E_{j,1}'(t).
\end{equation}
Therefore, the properties \eqref{eq4} for the function $g$, together to the fact that $E_{j,k}$, given in \eqref{enrgy_h}, is non-increasing and \eqref{eq29} give us
\begin{equation*}
\begin{split}
\|\Delta^\frac{j}{2} y\|^2\leq & 2c_{1,j}\|\eta^t\|_{L_j}^2+2c_{2,j}c_\epsilon \int_{0}^\infty g(s)\|\Delta^\frac{j}{2} \eta^t_{tt}\|^2ds -2c_{2,j}c_\epsilon E_{j,1}'(t)\\
\leq&  c_{j,4}\left( E_j(t)+E_{j,1}(t)+ E_{j,2}(t)\right) \\
\leq & c\left( E_j(0)+E_{j,1}(0)+ E_{j,2}(0)\right)
\end{split}
\end{equation*}
where $c=c(j):=c_{j,4}=\max\{4c_{1,j},4c_{2,j}c_\epsilon,2c_0c_{2,j}c_\epsilon\}$, for $j\in\{1,2\}$, proving the lemma.
\end{proof}

Before presenting the main result of this section, the next result ensures that the following norms  $\|\Delta^\frac{j}{2} \eta^t\|$,  $\|\eta^t\|$, and $\|\eta^t_{tt}\|$ can be controlled by the generalized energies $E_{k,j}(0)$ and the initial condition $y_0$, for $t\geq s\geq 0$. The result is the following one.
\begin{lemma} Considering the hypothesis of the Lemma \ref{lemma6},  the following inequality holds
\begin{equation}
\label{eq44}
\|\Delta^\frac{j}{2} \eta^t\|^2\leq M_{j,0}(t,s),
\end{equation}
where
\begin{equation}
\label{eq45}
M_{j,0}(t,s):=\left\{\begin{array}{l}
c\left(E_j(0)+E_{j,1}(0)+ E_{j,2}(0)\right), \mbox{\ if \ } 0\leq s\leq t,\\
\\
{\displaystyle \left\|\int_{0}^{s-t} \Delta^\frac{j}{2} y_0(\cdot,\tau)d\tau\right\|^2+2s^2c\left(E_j(0)+E_{j,1}(0)+ E_{j,2}(0)\right), \mbox{\ if \ } s>t\geq 0}.
\end{array}
\right.
\end{equation}
Additionally,  for $j=0$, we have 
\begin{equation}
\label{eq46}
\|\eta^t\|^2\leq M_{0,0}(t,s):=\left\{\begin{array}{l}
 2s^2E_0(0), \mbox{\ if \ } 0\leq s\leq t,\\
\\
{\displaystyle 2\left\|\int_{0}^{s-t} y_0(\cdot,\tau)d\tau\right\|^2+4s^2E_0(0), \mbox{\ if \ } s>t\geq 0}
\end{array}
\right.
\end{equation}
and
\begin{equation}
\label{eq48}
\|\eta^t_{tt}\|^2\leq M_{0,2}(t,s):=\left\{\begin{array}{l}
 2s^2E_{0,2}(0), \mbox{\ if \ } 0\leq s\leq t,\\
\\
{\displaystyle 2\left\|\int_{0}^{s-t} \partial^2_\tau y_0(\cdot,\tau)d\tau\right\|^2+4s^2E_{0,2}(0), \mbox{\ if  \ } s>t\geq 0.}
\end{array}
\right.
\end{equation}
\end{lemma}
\begin{proof}
Let us first prove \eqref{eq44}. H\"older inequality and \eqref{eq43}, for$j\in\{1,2\}$, gives that\begin{equation*}
\begin{split}
\|\Delta^\frac{j}{2}\eta^t\|^2=&\left\|\int_{t-s}^t \Delta^\frac{j}{2} y(\cdot, \tau)d\tau\right\|^2\leq\left(\int_{t-s}^t 1\cdot \|\Delta^\frac{j}{2} y(\cdot, \tau)\|d\tau\right)^2\leq s\left(\int_{t-s}^t\|\Delta^\frac{j}{2} y(\cdot, \tau)\|^2d\tau\right)\\
\leq& s^2c\left(E_j(0)+E_{j,1}(0)+ E_{j,2}(0)\right),
\end{split}
\end{equation*} 
for $t\geq s\geq 0.$ Analogously, 
\begin{equation*}
\begin{split}
\|\Delta^\frac{j}{2} \eta^t\|^2=\left\|\int_{t-s}^t \Delta^\frac{j}{2} y(\cdot,\tau)d\tau\right\|^2
\leq2\left\|\int_{0}^{s-t} \Delta^\frac{j}{2} y_0(\cdot,\tau)d\tau\right\|^2+2s^2c\left(E_j(0)+E_{j,1}(0)+ E_{j,2}(0)\right),
\end{split}
\end{equation*}
when $s>t\geq 0.$  Consequently,   \eqref{eq44} is verified. 

Now, for $j=0,$ since $\|y\|^2$ is part of $E_0$ (see \eqref{energy}),  and the energy $E_0$ is non-increasing,  we observe,  using H\"older inequality,  that
\begin{equation*}
\begin{split}
\|\eta^t\|^2=&\left\|\int_{t-s}^{t} y(\cdot,\tau)d\tau\right\|^2\leq\left(\int_{t-s}^{t} 1\cdot\|y(\cdot,\tau)\|d\tau\right)^2 \leq s\int_{t-s}^{t} \|y(\cdot,\tau)\|^2d\tau\\
\leq&s\int_{t-s}^{t} 2E_0(\tau)d\tau\leq s\int_{t-s}^{t} 2E_0(0)d\tau=2s^2E_0(0),\\
\end{split}
\end{equation*}
for $t\geq s\geq 0.$ On the other hand, 
\begin{equation*}
\begin{split}
\|\eta^t\|^2=&\left\|\int_{0}^{s-t} y_0(\cdot,\tau)d\tau+\int_{0}^{t} y(\cdot,\tau)d\tau\right\|^2
\leq2\left\|\int_{0}^{s-t} y_0(\cdot,\tau)d\tau\right\|^2+2\left\|\int_{0}^{t} y(\cdot,\tau)d\tau\right\|^2\\\
\leq&2\left\|\int_{0}^{s-t} y_0(\cdot,\tau)d\tau\right\|^2+4E_0(0)s^2,
\end{split}
\end{equation*}
for  $s>t\geq 0$.  Thus,  \eqref{eq46} follows.

Finally, let us prove \eqref{eq48}.  To do that, observe that \eqref{eq8} is linear and $V=\partial^2_tU$ is solution for \eqref{eq8} with initial condition $V(0)(x,s)=(\partial^2_t y_0(x,0),\zeta^0(x,s)),$ where ${\displaystyle \zeta^0(x,s)=\int_0^s\partial^2_\tau y_0(x,\tau)d\tau.}$ Thanks to relation  \eqref{eq45},  for $j\in\{1,2\}$, we get that
\begin{equation}
\label{eq47}
M_{j,2}(t,s):=\left\{\begin{array}{l}
 c_{j,5}\left(E_{j,2}(0)+E_{j,3}(0)+ E_{j,4}(0)\right), \mbox{\ if \ } 0\leq s\leq t,\\
\\
{\displaystyle 2\left\|\int_{0}^{s-t} \Delta^\frac{j}{2} \partial^2_\tau y_0(\cdot,\tau)d\tau\right\|^2}\\
\ \ {\displaystyle \quad \quad+2s^2c_{j,5}\left(E_{j,2}(0)+E_{j,3}(0)+ E_{j,4}(0)\right), \mbox{\ if \ } s>t\geq 0,}
\end{array}
\right.
\end{equation}
and so,
$$
\|\eta^t_{tt}\|^2\leq M_{j,2}(t,s).
$$
Therefore,  inequality \eqref{eq48} follows using the previous inequality with $j=0$, and thanks to the relation \eqref{eq46},  the result is proved. 
\end{proof}
The next result is the key lemma to establish the stabilization result for the Biharmonic Schr\"odinger system \eqref{eq1}. 
\begin{lemma} There exist positive constants $d_{j,k}, $ for each $j\in\{0,1,2\}$ and each $k\in\{0,2\}$ such that the following inequality holds 
\begin{equation}
\label{eq49}
\dfrac{G_0(\epsilon_0E_j(t))}{\epsilon_0E_j(t)}{\displaystyle \int_0^\infty g(s)\|\Delta^\frac{j}{2}\partial_t^k\eta^t\|^2ds}\leq -d_{j,k}E_{j,k}'(t)+d_{j,k} G_0(\epsilon_0E_j(t)),
\end{equation}
for any $\epsilon_0>0$.  Here,  $E_{j,0}=E_j, E'_{j,0}=E'_j(0)$ and $G_0$ defined as in Theorem \ref{teo2}.
\end{lemma}
\begin{proof}
Suppose, first, that the relation \eqref{eq21} is satisfied. So, thanks to the relation \eqref{eq29}, we have 
$$E'_{j,k}=\dfrac{1}{2}\int_0^\infty g'(s)\|\Delta^\frac{j}{2}\partial_t^k\eta^t\|^2ds\leq -\dfrac{1}{2}\alpha_0\int_0^\infty g(s)\|\Delta^\frac{j}{2}\partial_t^k\eta^t\|^2ds,$$
for each $j\in\{0,1,2\}$ and each $k\in\{0,2\}$,  that is, 
$$ \int_0^\infty g(s)\|\Delta^\frac{j}{2}\partial_t^k\eta^t\|^2ds\leq -\dfrac{2}{\alpha_0}E'_{j,k},
$$
showing \eqref{eq49} for each $d_{j,k}=\dfrac{2}{\alpha_0}$ and $G_0(s)=s.$ 

On the other hand, suppose now that \eqref{eq22} and \eqref{eq5} are verified.  Let us assume, without loss of generality, that $E_j(t)> 0$ and $g'<0$ in $\R_{+}$.   Let $\tau_{j,k}(t,s),\theta_j(t,s), j\in\{0,1,2\},\ k\in\{0,2\}$ and $\epsilon_0$ be a positive real number which will be fixed later on, and $K(s)=\dfrac{s}{G^{-1}(s)},$ for $s>0.$ Assumption \ref{H1} implies that $$\lim_{s \rightarrow 0^+}K(s)=\lim_{s \rightarrow 0^+}\dfrac{s}{G^{-1}(s)}=\lim_{\tau=G^{-1}(s) \rightarrow 0^+}\dfrac{G(\tau)}{\tau}=G'(0)=0.$$ Additionally, thanks to the continuity of $K$ we have $K(0)=0.$ 

We claim tha the function $K$ is non-decreasing. Indeed, since $G$ is convex we have that  $G^{-1}$  is concave and $G^{-1}(0)=0,$  implying that 
\begin{equation*}
K(s_1)= \dfrac{s_1}{G^{-1}\left(\frac{s_1}{s_2}s_2+\left(1-\frac{s_1}{s_2}\right)\cdot 0\right)}\leq \dfrac{s_1}{\frac{s_1}{s_2}G^{-1}\left(s_2\right)}=\dfrac{s_2}{G^{-1}\left(s_2\right)}=K(s_2),
\end{equation*}
for $0\leq s_1<s_2$, proving the claim.

Now, note that thanks to the fact that $K$ is non-decreasing and by \eqref{eq44}, \eqref{eq46}, \eqref{eq47}, and \eqref{eq48},  we get
\begin{equation}
\label{eq50}
K\left(-\theta_{j,k}(t,s)g'(s)\|\Delta^\frac{j}{2} \partial^k_t\eta^t\|^2\right)\leq K\left(-\theta_{j,k}(t,s)g'(s)M_{j,k}(t,s)\right).
\end{equation}
Inequality \eqref{eq50} yields that
\begin{equation}
\label{eq51}
\begin{split}
 \int_0^\infty g(s)\|\Delta^\frac{j}{2} \partial^k_t\eta^t\|^2 ds=&\dfrac{1}{G'\left(\epsilon_0 E_j(t)\right)}\int_0^\infty \dfrac{1}{\tau_{j,k}(t,s)}G^{-1}(-\theta_{j,k}g'(s)\|\Delta^\frac{j}{2} \partial^k_t\eta^t\|^2)\\
& \times\dfrac{\tau_{j,k}(t,s)G'\left(\epsilon_0 E_j(t)\right)g(s)}{-\theta_{j,k} g'(s)}K\left(-\theta_{j,k}g'(s)\|\Delta^\frac{j}{2}\partial^k_t \eta^t\|^2\right)ds\\
\leq& \dfrac{1}{G'\left(\epsilon_0 E_j(t)\right)}\int_0^\infty \dfrac{1}{\tau_{j,k}(t,s)}G^{-1}(-\theta_{j,k}g'(s)\|\Delta^\frac{j}{2} \partial^k_t\eta^t\|^2)\\
& \times\dfrac{\tau_{j,k}(t,s)G'\left(\epsilon_0 E_j(t)\right)g(s)}{-\theta_{j,k} g'(s)}K\left(-M_{j,k}\theta_{j,k}g'(s)\right)ds\\
\leq& \dfrac{1}{G'\left(\epsilon_0 E_j(t)\right)}\int_0^\infty \dfrac{1}{\tau_{j,k}(t,s)}G^{-1}(-\theta_{j,k}g'(s)\|\Delta^\frac{j}{2}\partial^k_t \eta^t\|^2)\\
& \times\dfrac{M_{j,k}(t,s)\tau_{j,k}(t,s)G'\left(\epsilon_0 E_j(t)\right)g(s)}{G^{-1}\left(-M_{j,k}\theta_{j,k}g'(s)\right)}ds.
\end{split}
\end{equation}

Denote the dual function of $G$ by $G^*(s)=\sup_{\tau\in\R_+}\{s\tau-G(\tau)\},$ for $s\in\R_+$.  From the Assumption \ref{H1} we have $$G^*(s)=s(G')^{-1}(s)-G((G')^{-1}(s)), \ s\in \R_+.$$ Observe also that 
$$
s_1s_2\leq G(s_1)+G^*(s_2),\ \forall s_1,s_2\in\R_+,
$$
in particular 
$$s_1=G^{-1}\left(-\theta_{j,k}(t,s)g'(s)\|\Delta^\frac{j}{2} \partial^k_t\eta^t\|^2\right)$$
and
$$s_2=\dfrac{M_{j,k}\tau_{j,k}G'(\epsilon_0 E_j(t))g(s)}{-M_{j,k}(t,s)g'(s)\theta_{j,k}}.$$
Therefore,  we obtain, by using the previous equality in \eqref{eq51}, that
\begin{equation*}
\begin{split}
 \int_0^\infty g(s)\|\Delta^\frac{j}{2} \partial^k_t\eta^t\|^2 ds \leq& \dfrac{1}{G'\left(\epsilon_0 E_j(t)\right)}\int_0^\infty \dfrac{1}{\tau_{j,k}(t,s)}\left(-\theta_{j,k}g'(s)\|\Delta^\frac{j}{2} \partial^k_t\eta^t\|^2\right)ds\\
&+ \dfrac{1}{G'\left(\epsilon_0 E_j(t)\right)}\int_0^\infty \dfrac{1}{\tau_{j,k}}G^*\left(\dfrac{M_{j,k}\tau_{j,k}G'\left(\epsilon_0 E_j(t)\right)g(s)}{G^{-1}\left(-M_{j,k}\theta_{j,k}g'(s)\right)}\right)ds.
\end{split}
\end{equation*}
Using that $G^*(s)\leq s(G')^{-1}(s),$  we get
\begin{equation*}
\begin{split}
\int_0^\infty g(s)\|\Delta^\frac{j}{2} \partial^k_t\eta^t\|^2 ds\leq&\dfrac{1}{G'\left(\epsilon_0 E_j(t)\right)}\left[\int_0^\infty \dfrac{1}{\tau_{j,k}}\left(-\theta_{j,k}g'(s)\|\Delta^\frac{j}{2}\partial^k_t \eta^t\|^2\right)ds\right.\\
&+\left.\int_0^\infty \dfrac{1}{\tau_{j,k}}\dfrac{M_{j,k}\tau_{j,k}G'\left(\epsilon_0 E_j(t)\right)g(s)}{G^{-1}\left(-M_{j,k}\theta_{j,k}g'(s)\right)}(G')^{-1}\left(\dfrac{M_{j,k}\tau_{j,k}G'\left(\epsilon_0 E_j(t)\right)g(s)}{G^{-1}\left(-M_{j,k}\theta_{j,k}g'(s)\right)}\right)ds\right].
\end{split}
\end{equation*}
Pick $\theta_{j,k}=\frac{1}{M_{j,k}},$  to ensure that
\begin{equation*}
\begin{split}
 \int_0^\infty g(s)\|\Delta^\frac{j}{2}\partial^k_t \eta^t\|^2 ds\leq&  \dfrac{1}{G'\left(\epsilon_0 E_j(t)\right)}\int_0^\infty \dfrac{1}{\tau_{j,k}M_{j,k}}\left(-g'(s)\|\Delta^\frac{j}{2}\partial^k_t \eta^t\|^2\right)ds\\& + \int_0^\infty \dfrac{M_{j,k}g(s)}{G^{-1}\left(-g'(s)\right)}(G')^{-1}\left(\dfrac{M_{j,k}\tau_{j,k}G'\left(\epsilon_0 E_j(t)\right)g(s)}{G^{-1}\left(-g'(s)\right)}\right)ds.
\end{split}
\end{equation*}
Thanks to the fact that $(G')^{-1}$ is non-decreasing we get, 
\begin{equation*}
\begin{split}
\int_0^\infty g(s)\|\Delta^\frac{j}{2} \partial^k_t\eta^t\|^2 ds\leq& \dfrac{1}{G'\left(\epsilon_0 E_j(t)\right)}\int_0^\infty \dfrac{1}{\tau_{j,k}M_{j,k}}\left(-g'(s)\|\Delta^\frac{j}{2}\partial^k_t \eta^t\|^2\right)ds\\
&+ \int_0^\infty \dfrac{M_{j,k}g(s)}{G^{-1}\left(-g'(s)\right)}(G')^{-1}\left(m_0M_{j,k}\tau_{j,k}G'\left(\epsilon_0 E_j(t)\right)\right)ds,
\end{split}
\end{equation*}
where $m_0=\sup_{s\in\R_+}\dfrac{g(s)}{G^{-1}\left(-g'(s)\right)}$.  Note that \eqref{eq22} and \eqref{eq23},  yields that
$$m_1=\sup_{s\in\R_+}\int_0^\infty\dfrac{M_{j,k}(s,t)g(s)}{G^{-1}\left(-g'(s)\right)}ds<\infty.$$
Thus, using that  $\tau_{j,k}(t,s)=\frac{1}{m_0M_{j,k}(t,s)}$ and relation \eqref{eq28},  we have that
\begin{equation*}
\begin{split}
 \int_0^\infty g(s)\|\Delta^\frac{j}{2} \partial^k_t\eta^t\|^2 ds\leq& -\dfrac{m_0}{G'\left(\epsilon_0 E_j(t)\right)}\int_0^\infty \left(g'(s)\|\Delta^\frac{j}{2} \partial^k_t\eta^t\|^2\right)ds\\&+ \left(\epsilon_0 E_j(t)\right)\int_0^\infty \dfrac{M_{j,k}g(s)}{G^{-1}\left(-g'(s)\right)}ds\\
=& -\dfrac{2m_0}{G'\left(\epsilon_0 E_j(t)\right)}E_{j,k}'(t)
+ \epsilon_0 m_1 E_j(t).
\end{split}
\end{equation*}
Finally, multiplying the previous inequality by $G'\left(\epsilon_0 E_j(t)\right)=\dfrac{G_0\left(\epsilon_0 E_j(t)\right)}{\epsilon_0 E_j(t)}$ gives
$$\dfrac{G_0\left(\epsilon_0 E_j(t)\right)}{\epsilon_0 E_j(t)}\int_0^\infty g(s)\|\Delta^\frac{j}{2}\partial^k_t \eta^t\|^2 ds\leq -2m_0E_{j,k}'(t)+  m_1G_0\left(\epsilon_0 E_j(t)\right),
$$
which taking $d_{j,k}=\max\{2m_0,m_1\},$ ensures \eqref{eq49}, showing the lemma. 
\end{proof}

\section{Proof of Theorem \ref{teo2}}\label{sec4}
Let us split the proof into two cases: a) $n=1$ and b) $n>1$.

\begin{itemize}
\item[a)] $n=1$
\end{itemize}

Poincar\'e inequality gives us 
$$\|y\|^2\leq c_*\|\nabla y\|^2\leq  c_*^2\|\Delta y\|^2,$$
where $c_*>0$ is the Poincar\'e constant.  Summarizing, 
$$
\|y\|^2\leq  c_{**}\|\Delta^\frac{j}{2} y\|^2,
$$
for $c_{**}$ defined by \eqref{c*}.  From the definition of $E_j$ given by \eqref{energy} we found that
$$
 \dfrac{2}{\epsilon_0c_{**}}G_0(\epsilon_0E_j(t))\leq \dfrac{G_0(\epsilon_0E_j(t))}{\epsilon_0E_j(t)} \|\Delta^\frac{j}{2} y\|^2+\dfrac{1}{c_{**}}\dfrac{G_0(\epsilon_0E_j(t))}{\epsilon_0E_j(t)}\int_0^\infty g(s)\|\Delta^\frac{j}{2} \eta^t\|^2ds.
$$
Thanks to the inequality \eqref{eq42}, we have
\begin{equation}
\label{eq52}
\begin{split}
\dfrac{2}{\epsilon_0c_{**}}G_0(\epsilon_0E_j(t) \leq&\dfrac{G_0(\epsilon_0E_j(t))}{\epsilon_0E_j(t)}  \left(2c_{1,j}\|\eta^t\|_{L_j}^2+2c_{2,j}c_\epsilon \int_{0}^\infty g(s)\|\Delta^\frac{j}{2} \eta^t_{tt}\|^2ds -2c_{2,j}c_\epsilon E_{j,1}'(t)\right)\\&+\dfrac{1}{c_{**}}\dfrac{G_0(\epsilon_0E_j(t))}{\epsilon_0E_j(t)}\int_0^\infty g(s)\|\Delta^\frac{j}{2} \eta^t\|^2ds\\
=&\dfrac{G_0(\epsilon_0E_j(t))}{\epsilon_0E_j(t)} \left(2c_{2,j}c_\epsilon \int_{0}^\infty g(s)\|\Delta^\frac{j}{2} \eta^t_{tt}\|^2ds-2c_{2,j}c_\epsilon E_{j,1}'(t)\right)\\& +\left(2c_{1,j}+\dfrac{1}{c_{**}}\right)\dfrac{G_0(\epsilon_0E_j(t))}{\epsilon_0E_j(t)}\int_0^\infty g(s)\|\Delta^\frac{j}{2} \eta^t\|^2ds.
\end{split}
\end{equation}
Combining \eqref{eq52} with \eqref{eq49}, gives
\begin{equation*}
\begin{split}
\dfrac{2}{\epsilon_0c_{**}}G_0(\epsilon_0E_j(t))\leq &  -2c_{2,j}c_\epsilon d_{j,2}E_{j,2}'(t)+2c_{2,j}c_\epsilon d_{j,2} G_0(\epsilon_0E_j(t))-2c_{2,j}\dfrac{G_0(\epsilon_0E_j(t))}{\epsilon_0E_j(t)}c_\epsilon E_{j,1}'(t)\\
&-d_{j,0}\left(2c_{1,j}+\dfrac{1}{c_{**}}\right)E_{j}'(t)+d_{j,0}\left(2c_{1,j}+\dfrac{1}{c_{**}}\right) G_0(\epsilon_0E_j(t)).
\end{split}
\end{equation*}
So,
\begin{equation}
\label{eq53}
\begin{split}
\left(\dfrac{2}{\epsilon_0c_{**}}-2c_{2,j}c_\epsilon d_{j,2}-d_{j,0}\left(2c_{1,j}+\dfrac{1}{c_{**}}\right) \right)&G_0(\epsilon_0E_j(t))\leq  -2c_{2,j}c_\epsilon d_{j,2}E_{j,2}'(t)  \\&-2c_{2,j}\dfrac{G_0(\epsilon_0E_j(t))}{\epsilon_0E_j(t)}c_\epsilon E_{j,1}'(t) -d_{j,0}\left(2c_{1,j}+\dfrac{1}{c_{**}}\right)E_{j}'(t).\\
\end{split}
\end{equation}
Observe that $H_0(s)=\dfrac{G_0(s)}{s}$ is non-decreasing and $E_j$ is non-increasing for each $j$, thus $\dfrac{G_0(\epsilon_0E_j(t))}{\epsilon_0E_j(t)}$ is non-increasing for each $j$, and therefore by \eqref{eq53} we get
\begin{equation}
\label{eq54}
\begin{split}
\left(\dfrac{2}{\epsilon_0c_{**}}-2c_{2,j}c_\epsilon d_{j,2}-d_{j,0}\left(2c_{1,j}+\dfrac{1}{c_{**}}\right) \right)&G_0(\epsilon_0E_j(t))\leq  -2c_{2,j}c_\epsilon d_{j,2}E_{j,2}'(t)\\ &-2c_{2,j}\dfrac{G_0(\epsilon_0E_j(0))}{\epsilon_0E_j(0)}c_\epsilon E_{j,1}'(t)-d_{j,0}\left(2c_{1,j}+\dfrac{1}{c_{**}}\right)E_{j}'(t).
\end{split}
\end{equation}
For $\epsilon_0>0$ small enough  we have $$c_1=\left(\dfrac{2}{\epsilon_0c_{**}}-2c_{2,j}c_\epsilon d_{j,2}-d_{j,0}\left(2c_{1,j}+\dfrac{1}{c_{**}}\right)\right)>0.$$ Thus, dividing  \eqref{eq54} by $c_1>0$ yields that
\begin{equation}
\label{eq55}
G_0(\epsilon_0E_j(t))\leq   -c_2\left(E_{j}'(t)+ E_{j,1}'(t)+E_{j,2}'(t)\right),
\end{equation}
where 
 $$c_2=\max\left\{\dfrac{2c_{2,j}c_\epsilon d_{j,2}}{c_1},\dfrac{2c_{2,j}\dfrac{G_0(\epsilon_0E_j(0))}{\epsilon_0E_j(0)}c_\epsilon}{c_1},\dfrac{d_{j,0}\left(2c_{1,j}+\dfrac{1}{c_{**}}\right)}{c_1}\right\}. $$
 Now, integrating \eqref{eq55} on $[0,t], \ t\in\R_+^*,$ and observing that  $G_0(\epsilon_0E_j(t))$ is non-increasing gives
\begin{equation*}
\begin{split}
tG_0(\epsilon_0E_j(t))=&  \int_0^t G_0(\epsilon_0E_j(t))ds\leq \int_0^t G_0(\epsilon_0E_j(s))ds
\leq   -c_2\int_0^t \left(E_{j}'(s)+ E_{j,1}'(s)+E_{j,2}'(s)\right)ds\\
=&  -c_2\left(E_{j}(t)+ E_{j,1}(t)+E_{j,2}(t)\right) +c_2\left(E_{j}(0)+ E_{j,1}(0)+E_{j,2}(0)\right)\\
\leq& c_2\left(E_{j}(0)+ E_{j,1}(0)+E_{j,2}(0)\right)=:c_3.
\end{split}
\end{equation*}
Due to the fact that $G_0$ is inversible and non-decreasing, we deduce that
$$E_j(t)\leq \dfrac{1}{\epsilon_0}(G_0)^{-1}\left(\dfrac{c_3}{t}\right)= \dfrac{1}{\epsilon_0}G_1\left(\dfrac{c_3}{t}\right)\leq \alpha_{j,1}G_1\left(\dfrac{\alpha_{j,1}}{t}\right),$$
for $\alpha_{j,1}=\max\left\{\dfrac{1}{\epsilon_0},c_3\right\},$ showing \eqref{eq24} when $n=1.$

\begin{itemize}
\item[a)] $n>1$
\end{itemize}
Suppose, for induction hypothesis, that for some $n\in\mathbb{N}^*,$ we have that \eqref{eq24} is verified when  $U_0\in D(\mathcal{A}_j^{2n+2})$ for $j\in\{1,2\}$  and $U_0\in D(\mathcal{A}_j^{2n})$ for $j=0.$  For $j\in\{1,2\}$, let us take $U_0\in D(\mathcal{A}_j^{2(n+1)+2})$ and for $j=0,$  take $U_0\in D(\mathcal{A}_j^{2(n+1)}).$ So when $j\in\{1,2\}$ we have $$U_0\in D(\mathcal{A} _j^{2(n+1)+2})\subset D(\mathcal{A }_j^{2n+2}), \quad U_t(0)\in D(\mathcal{A}_j^{2 (n+1)+1})\subset D(\mathcal{A}_j^{ 2n+2}), \quad \text{ and } U_{tt}(0)\in D(\mathcal{A}_j^{2n+ 2}).$$  Now, for  $j=0,$  we found $$U_0\in D(\mathcal{A}_j^{2(n+1)})\subset D(\mathcal{A}_j^{2n}), \quad U_t(0)\in D(\mathcal{A}_j^{2n+1})\subset D(\mathcal{A}_j^{2n}), \quad \text{ and }U_{tt}(0)\in D(\mathcal{A}_j^{2n}).$$ So, follows from the induction hypothesis that: there exists $\alpha_{j,n}$  such that $$E_j(t)\leq \alpha_{j,n} G_n\left(\dfrac{\alpha_{j,n}}{t}\right), \forall t\in \R_+^*.$$

Now,  since  $U_t$ and $U_{tt}$ are solution of \eqref{eq8} with initial conditions  $U_t(0)\in D(\mathcal{A}_j^{2n+2})$ and  $U_{tt}(0)\in D(\mathcal{A}_j^{2n+2})$, respectively,  the induction hypothesis guarantees the existence of $\beta_{n,t}>0$ and $\gamma_{n,t}>0$, such that 
$$E_{j,1}(t)\leq \beta_{j,n} G_n\left(\dfrac{\beta_{j,n}}{t}\right), \forall t\in \R_+^* \quad \text{and} \quad E_{j,2}(t)\leq \gamma_{j,n} G_n\left(\dfrac{\gamma_{j,n}}{t}\right), \forall t\in \R_+^*,$$
respectively.  Thus,  as $G_n's$ are non-decreasing for $\tilde{d}_{j,n}=\max\{3\alpha_{j,n},3\beta_{j,n},3\gamma_{j,n}\},$ we get 
$$
E_j(t)+E_{j,1}(t)+E_{j,2}(t)\leq \tilde{d}_{j,n}  G_n\left(\dfrac{\tilde{d}_{j,n}}{t}\right).
$$
Finally, how $t\in[T,2T]$,  we have
$$G_0(\epsilon_0E_j(2T))\leq G_0(\epsilon_0E_j(t))$$
and from  \eqref{eq55} we found the following
\begin{equation*}
\begin{split}
TG_0(\epsilon_0E_j(2T))\leq&  \int_T^{2T}G_0(\epsilon_0E_j(t))dt\leq -c_2\int_T^{2T}\left(E_j'(t)+E_{j,1}'(t)+E_{j,2}'(t)\right)dt\\
=& -c_2\left(E_j(2T)+E_{j,1}(2T)+E_{j,2}(2T)\right)+c_2\left(E_j(T)+E_{j,1}(T)+E_{j,2}(T)\right)\\
\leq& c_2\left(E_j(T)+E_{j,1}(T)+E_{j,2}(T)\right)\leq c_2\tilde{d}_{j,n}  G_n\left(\dfrac{\tilde{d}_{j,n}}{T}\right)\leq  d_{j,n}  G_n\left(\dfrac{d_{j,n}}{T}\right),
\end{split}
\end{equation*}
where  $d_{j,n}=\max\{c_2\tilde{d}_{j,n},\tilde{d}_{j,n}\}.$ Moreover,  as $G_0$ is non-decreasing,  $G_1=G_0^{-1}$ is also non-decreasing.  Therefore, 
$$E_j(2T)\leq \dfrac{1}{\epsilon_0}G_0^{-1}\left(\dfrac{2d_{j,n}}{2T}  G_n\left(\dfrac{2d_{j,n}}{2T}\right)\right)=\dfrac{1}{\epsilon_0}G_1\left(\tilde{s}  G_n\left(\tilde{s}\right)\right)=\dfrac{1}{\epsilon_0} G_{n+1}\left(\tilde{s}\right)=\alpha_{j,n+1} G_{n+1}\left(\dfrac{\alpha_{j,n+1}}{2T}\right),$$
where $\alpha_{j,n+1}:=\max\left\{\dfrac{1}{\epsilon_0},2d_{j,n}\right\}.$ In other words, there is $\alpha_{j,n+1}>0$ such that \eqref{eq24} holds for  $n+1.$ By the principle of induction we have that \eqref{eq24} is verified for all $n\in \mathbb{N}^*$, showing  Theorem \ref{teo2}.\qed

\appendix
\section{Well-posedness \textit{via} Semigroup theory} \label{sec2}
 This section is devoted to proving that the system \eqref{eq8} is well-posed in the energy space $\mathcal{H}_j$.  To do that, first, let us present some properties of $\mathcal{A}_j$, defined by \eqref{eq9}-\eqref{eq9a} and its adjoin $\mathcal{A}^*_j$ defined by
 \begin{equation}
\label{eq15}
\mathcal{A}_j^*(V)=\left(
\begin{array}{c}
-i\Delta v+i\Delta^2 v+(-1)^{j}\int_0^\infty g(s)\Delta^j\zeta^t(\cdot,s)ds\\
\\
\zeta^t_s+\dfrac{g'(s)}{g(s)}\zeta^t-v
\end{array}\right)
\end{equation}
with
\begin{equation}
\label{eq16}D(\mathcal{A}_j^*)=\{V=(v,\zeta^t)\in \mathcal{H}_j; \mathcal{A^*}_j(V)\in\mathcal{H}_j, v\in H^2_0(\Omega), \zeta^t(x,0)=0 \},
\end{equation}
for $j\in \{0,1,2\}$.  So,  our first result in this section ensures that  $\mathcal{A}_j$ (resp.  $\mathcal{A}^*_j$) is dissipative,  and $D(\mathcal{A}_j)$ (resp.  $D(\mathcal{A}^*_j)$) is dense in the energy space\footnote{Now on,  we will use the following Poincar\'e inequality 
$\|y\|^2\leq c_*\|\nabla y\|^2, \  \ y\in H_0^1(\Omega),$
where $c_*>0$ is the Poincar\'e constant. }. 
\begin{lemma}\label{lem301}
$\mathcal{A}_j$ and $\mathcal{A}^*_j$ are dissipative.  Moreover,  $D(\mathcal{A}_j)$ and $D(\mathcal{A}^*_j)$ are dense in $\mathcal{H}_j,$ for $j\in\{0,1,2\}$.
\end{lemma}
\begin{proof}
Indeed,  let $(y,\eta^t)\in D(\mathcal{A}_j)$ so
$$\langle \mathcal{A}_j(y,\eta^t),( y,\eta^t)\rangle=-\re\left(\int_0^\infty g(s)\int_\Omega\Delta^\frac{j}{2}  \eta^t_s\Delta^\frac{j}{2} \overline{\eta^t}dxds\right). $$
As
$$\Delta^\frac{j}{2} \eta^t_s\Delta^\frac{j}{2} \overline{\eta^t}=\dfrac{1}{2}(|\Delta^\frac{j}{2} \eta^t|^2)_s+i\im(\Delta^\frac{j}{2} \eta^t_s\Delta^\frac{j}{2} \overline{\eta^t}),$$
integration by parts over variable $s,$ ensures that
\begin{equation}
\label{eq10}
\begin{split}
\langle \mathcal{A}_j(y,\eta^t),( y,\eta^t)\rangle=&-\re\left(\int_0^\infty g(s)\int_\Omega \left(\dfrac{1}{2}(|\Delta^\frac{j}{2} \eta^t|^2)_s+i\im(\Delta^\frac{j}{2} \eta^t_s\Delta^\frac{j}{2} \overline{\eta^t})\right)dxds\right)\\
=&\dfrac{1}{2}\re\left( \int_0^\infty g'(s)\int_\Omega |\Delta^\frac{j}{2} \eta^t|^2dxds\right)\\
=&\dfrac{1}{2}\int_0^\infty g'(s)\|\Delta^\frac{j}{2} \eta^t\|^2ds\leq 0,\\
\end{split}
\end{equation}
since \eqref{eq4} is verified.  So,  $\mathcal{A}_j$ is dissipative.  Similarly,  $\mathcal{A}_j^*$ defined by \eqref{eq15} is dissipative. 

Now, let us prove that $D(\mathcal{A}_j)$ is dense on $\mathcal{H}_j.$  Since we showed that $\mathcal{A}_j$ is dissipative,  we need to prove that the image of $I-\mathcal{A}_j$  is  $\mathcal{H}_j,$ since $\mathcal{H}_j$ is reflexive. To do that, pick $(f_1,f_2)\in \mathcal{H}_j=L^2(\Omega)\times L^2_g(\R_+;H_0^j(\Omega))$,  we claim that there exists $(y,\eta^t)\in D(\mathcal{A}_j)$ such that $$(y,\eta^t)-(i\Delta y-i\Delta^2y+(-1)^{j+1}\int_0^\infty g(s)\Delta^j\eta^t(\cdot, s)ds, y-\eta_s^t)=(f_1,f_2).$$ Or equivalently,  we claim that there exits $(y,\eta^t)\in D(\mathcal{A}_j)$ satisfying
\begin{equation}
\label{eq11}
\left\{\begin{array}{l}
{\displaystyle y-i\Delta y+i\Delta^2y+(-1)^j\int_0^\infty g(s)\Delta^j\eta^t(\cdot, s)ds=f_1}\\
\eta^t-y+\eta_s^t=f_2.
\end{array}\right.
\end{equation}

Indeed,  multiplying the second equation of \eqref{eq11} by $e^s$ and integrating over $s,$ we get 
\begin{equation}
\label{eq12}
\begin{array}{rcl}
{\displaystyle \eta^t(x,s)}&=&{\displaystyle (1- e^{-s})y+\int_0^s e^{\tau-s} f_2(\tau)d\tau=(1- e^{-s})y+ f_3(s)}.
\end{array}
\end{equation}
Since $f_2\in L_g^2(\R_+;H_0^j(\Omega)),$ taking ${\displaystyle f_3=\int_0^s e^{\tau-s} f_2(\tau)d\tau}$ we have
\begin{equation*}
\begin{split}
\int_0^\infty g(s)\|\Delta^\frac{j}{2} f_3(s)\|^2ds=&\int_0^\infty g(s)e^{-2s}\int_{\Omega}\left|\int_0^s e^{\tau} \Delta^\frac{j}{2} f_2(\tau)d\tau\right|^2 dxds\\
\leq& \int_0^\infty g(s)e^{-s}\int_{\Omega}\int_0^s e^{\tau} |\Delta^\frac{j}{2} f_2(\tau)|^2d\tau dxds\\
\leq& \int_0^\infty \int_0^s g(s)e^{-s} e^{\tau} \|\Delta^\frac{j}{2} f_2(\tau)\|^2 d\tau ds\\
=&\int_0^\infty \int_\tau^\infty g(s)e^{-s} e^{\tau} \|\Delta^\frac{j}{2} f_2(\tau)\|^2  dsd\tau\\
\leq& \int_0^\infty \int_\tau^\infty g(\tau)e^{-s} e^{\tau} \|\Delta^\frac{j}{2} f_2(\tau)\|^2  dsd\tau\\
=& \| f_2\|_{L_g^2(\R_+;H_0^j(\Omega))}^2<+\infty,
\end{split}
\end{equation*}
that is, $f_3\in L_g^2(\R_+;H_0^j(\Omega)).$ Now, for $y\in H^2_0(\Omega)$  holds that 
$$\int_0^\infty g(s)\|(1-e^{-s})\Delta^\frac{j}{2} y\|^2ds=\|\Delta^\frac{j}{2} y\|^2\int_0^\infty g(s)(1-e^{-s})^2ds\leq\|\Delta^\frac{j}{2} y\|^2g_1<+\infty,$$
since
$$g_1:=\int_0^\infty g(s)(1-e^{-s})ds\leq \int_0^\infty g(s)ds=g_0.$$
So $(1-e^{-s})y\in L_g^2(\R_+;H_0^j(\Omega)).$ Therefore,  for $y\in H_0^2(\Omega)$, choosing $\eta^t$ as in \eqref{eq12},  follows that $\eta^t\in  L_g^2(\R_+;H_0^j(\Omega))$ and, so $ \eta^t(x,0)=0$. Thanks to \eqref{eq11} we get $$\eta^t_s=f_2-\eta^t+y\in  L_g^2(\R_+;H_0^j(\Omega)).$$

Finally, let us prove that $y\in H_0^2(\Omega)$ satisfies
\begin{equation}
\label{eq13}
y-i\Delta y+i\Delta^2y+(-1)^j\int_0^\infty g(s)\Delta^j \eta^t(\cdot, s)ds=f_1,
\end{equation}
for $\eta^t=(1-e^{-s})y+f_3.$ This is equivalent to obtain $y\in H_0^2(\Omega)$ satisfying the following elliptic equation
\begin{equation}
\label{eq14}
y-i\Delta y+i\Delta^2y+(-1)^j g_1\Delta^j y=f_1-(-1)^j\int_0^\infty g(s)\Delta^j f_3(\cdot, s)ds, 
\end{equation}
which is a direct consequence of the Lax-Milgram theorem.  Therefore,  $(y,\eta^t)\in D(\mathcal{A}_j)$ is a strong solution of $(I-\mathcal{A}_j)(y,\eta^t)=(f_1,f_2)$ and 
$I-\mathcal{A}_j$ is surjective,  showing the result.  Similarly, it is shown that  $D(\mathcal{A}_j^*)$ defined by \eqref{eq16} is dense in $\mathcal{H}_j.$
\end{proof}

The main result of this section is a consequence of the Lemma \ref{lem301} and can be read as follows.
\begin{theorem}
\label{teo1}
Suppose that Assumption \ref{R1} and \eqref{eq3} are verified.  Thus,  for each $j\in\{0,1,2\},$ the linear operator $\mathcal{A}_j$ defined by \eqref{eq9} is the infinitesimal generator of a semigroup of class $C_0$ and, for each $n\in\mathbb{N}$ and $U_0\in D(\mathcal{A}_j^n),$ the system \eqref{eq8} has unique solution in the class $U\in \bigcap_{k=0}^{n}C^k(\R_+;D(\mathcal{A}^{n-k}_j)).$
\end{theorem}

\subsection*{Acknowledgments} 
Capistrano--Filho was supported by grants numbers CNPq 307808/2021-1,   401003/2022-1 and 200386/2022-0, CAPES 88881.311964/2018-01 and 88881.520205/2020-01,  and MATHAMSUD 21-MATH-03.  This work is part of the Ph.D. thesis of de Jesus at the Department of Mathematics of the Federal University of Pernambuco.


\end{document}